\documentclass[11pt,a4paper]{amsart}

\usepackage{amssymb}
\usepackage{amsthm}
\usepackage{amsmath}
\usepackage{graphicx}
\usepackage[hdivide={2.5cm,,2.5cm}, vdivide={3.5cm,,2.8cm}]{geometry} 

\usepackage[ 	colorlinks	=true, 
				linkcolor	=blue, 
				citecolor	=blue, 
				draft	=false, 
				bookmarks, 
				bookmarksnumbered	=true, 
				plainpages	=false
				]{hyperref}

\usepackage{comment}
\usepackage[T1]{fontenc}
\numberwithin{equation}{section}

\newtheorem{thm}{Theorem}[section]

\newtheorem{lem}[thm]{Lemma}
\newtheorem{prop}[thm]{Proposition}
\newtheorem{problem}[thm]{Problem}
\newtheorem{question}[thm]{Question}

\newtheorem*{claim}{Claim}

\theoremstyle{definition}

\newtheorem{rem}[thm]{Remark}

\newcommand{\CC}{\widehat{\mathbb{C}}}%riemann surface
\newcommand{\C}{\mathbb{C}}%complex plane
\newcommand{\D}{\mathbb{D}}%unit disk
\newcommand{\N}{\mathbb{N}}%natural number
\newcommand{\R}{\mathbb{R}}%real number
\newcommand{\A}{\mathcal{A}}%analytic functions
\renewcommand{\S}{\mathcal{S}}%univalent functions
\newcommand{\CTC}{\mathcal{C}}%clsoe-to-convex functions
\renewcommand{\R}{\mathcal{R}}%Noshiro-Warschawski functions
\newcommand{\dstyle}{\displaystyle}
\renewcommand{\Re}{\mathsf{Re}\,}
\renewcommand{\Im}{\mathsf{Im}\,}
\newcommand{\closure}{\overline}

\newcommand{\de}{\partial}

\usepackage{xcolor}
\definecolor{Ikkeicolor}{RGB}{220,0,0}
\title[Boundary geometry and linear accessibility]
{Boundary geometry and linear accessibility of functions with positive real derivative}
\author[S. Hoshinaga]{Shota Hoshinaga}
\author[I. Hotta]{Ikkei Hotta}
\author[L.-M. Wang]{Li-Mei Wang}
\thanks{The second author was supported by JSPS KAKENHI Grant Numbers JP25K07049 and JP23K25775.}
\address{National Institute of Technology, Kure College, 2-2-11 Aga-minami, Kure, Hiroshima 737-8506, Japan}
\email{s-hoshinaga@kure-nct.ac.jp}
\address{Department of Applied Science, Faculty of Engineering, Yamaguchi University, 2-16-1 Tokiwadai, Ube, Yamaguchi 755-8611, Japan}
\email{ihotta@yamaguchi-u.ac.jp}
\address{School of Statistics,
	University of International Business and Economics, No.~10, Huixin
	Dongjie, Chaoyang District, Beijing 100029, China}
\email{wangmabel@163.com}
\begin{document}

	%	++++++++++++++++++++++++++++++++++++++++++++++++++++++
	%
	%		Abstract
	%
	%	++++++++++++++++++++++++++++++++++++++++++++++++++++++

\begin{abstract}
We study the boundary geometry of the Noshiro--Warschawski class $\R$.
Using the geometric structure of close-to-convex domains and their
relation to Loewner chains, we investigate the boundary behavior of
functions in $\R$. In particular, we discuss the relation between
spherical length and local connectedness, and show that the boundary
of the image domain of a function in $\R$ need not be locally
connected. We also revisit the classical fact that $\R$ is not
contained in the class $\S^{*}$ of starlike functions and give a simple explicit
example of a function in $\R\setminus\S^*$.
\end{abstract}

\maketitle

	%	++++++++++++++++++++++++++++++++++++++++++++++++++++++
	%
	%		Section 1
			\section{Introduction}
	%
	%	++++++++++++++++++++++++++++++++++++++++++++++++++++++

One of the central themes in geometric function theory is to understand how analytic conditions imposed on a holomorphic function are reflected in the geometry of its image domain. While the local behavior of a non-constant holomorphic function is quite rigid, the boundary behavior of a conformal map can be extremely complicated. Even for normalized univalent functions on the unit disk, no reasonable regularity of the boundary can be expected in general.

This contrast becomes particularly interesting for classical subclasses of univalent functions. For instance, convex functions have locally connected boundaries, and hence the corresponding conformal maps extend continuously to the closed disk. Starlike and close-to-convex functions form larger and less rigid classes, and for such classes it becomes a delicate problem to determine how much boundary regularity is forced by the defining analytic condition.

In this paper, we investigate the boundary geometry of the
Noshiro--Warschawski class and discuss several related problems for classical subclasses of univalent functions.
Let $\D := \{z \in \C : |z|  <1\}$, and let $\A$ denote the class of holomorphic functions $f$ on $\D$ normalized by $f(0) =0$ and $f'(0)=1$.
We denote by $\S$ the subclass of $\A$ consisting of univalent functions.
We consider the class
\[
\R
:= \{f\in\A: \Re f'(z)>0 \text{ for all } z\in\D\}.
\]
It is known that for a function $f\in\mathcal A$, the condition that $f'(\D)$ is contained in a half-plane not containing the origin ensures the univalence of $f$ on $\D$. In particular, every function in $\R$ is univalent, and hence $\R\subset\mathcal S$. This is the classical Noshiro--Warschawski criterion, due independently to Noshiro \cite{Noshiro:1934} and Warschawski \cite{Warschawski:1935}.

As will be discussed in Section~2, close-to-convex functions admit a natural description in terms of Loewner chains. In particular, the complement of the image domain can be covered by a union of half-lines. For functions in $\R$, these half-lines can be chosen in a more specific way. This observation is one of the starting points of the present paper: we investigate to what extent such geometric restrictions are reflected in the boundary behavior of functions in $\R$.

The paper is organized as follows. 
In Section~2, we review close-to-convex functions and Loewner chains, and collect several facts needed later. We also discuss linear accessibility of close-to-convex domains. 
In Section~3, we study the boundary geometry of functions in $\R$. After discussing spherical length and its relation to local connectedness, we give a negative answer to the question of whether the boundary of every image domain of a function in $\R$ is locally connected. 
In Section~4, we revisit the classical fact that there exist functions in $\R$ which are not starlike, and give a particularly elementary example of such a function. 
Appendix~A collects some facts concerning the spherical distance and local connectedness that are used in Section~3, while Appendix~B is devoted to the local uniform convergence of linearly accessible functions.

\begin{rem}[General Conventions]
Throughout this paper, unless otherwise stated, the boundary of any plane domain is taken in the Riemann sphere $\CC$. Furthermore, continuity at boundary points is always understood with respect to the spherical metric.
\end{rem}

	%	++++++++++++++++++++++++++++++++++++++++++++++++++++++
	%
	%		Section 2
			\section{Preliminaries on close-to-convex functions and Loewner chains}
	%
	%	++++++++++++++++++++++++++++++++++++++++++++++++++++++

\subsection{Close-to-convex functions}

%Let $\mathcal S^*$ denote the class of normalized starlike functions in $\S$.
%, that is,
%$$
%\mathcal S^*
%:=
%\left\{
% f\in\mathcal S:
% \operatorname{Re}\frac{zf'(z)}{f(z)}>0
% \quad (z\in\D)
%\right\}.
%$$
As a subclass of $\mathcal S$, Kaplan introduced the class of close-to-convex functions in 1952 \cite{Kaplan:1952}. A function $f\in \A$ is said to be \textit{close-to-convex} if there exist a starlike function $g$ such that
\begin{equation}
\label{definition-ctc}
\Re
\frac{zf'(z)}{g(z)}
>0
\quad (z\in\D),
\end{equation}
Here, a holomorphic function $g$ on $\D$ with $g(0)=0$ is said to be \textit{starlike} if $\Re[ zg'(z)/g(z) ]>0$ for all $z \in \D$.
We denote by $\CTC$ the class of close-to-convex functions.
In particular, by choosing the starlike function $g(z)=z$, we see that
$\R \subset \CTC$. Thus the class $\R$ may be regarded as a special case of close-to-convex functions.

The image of $\D$ under a function $f\in\CTC$ is known to be \textit{linearly accessible}. More precisely, $\C\!\setminus \!f(\D)$ can be represented as a union of closed half-lines which are mutually disjoint except possibly at their endpoints. The notion of linear accessibility was introduced by Biernacki \cite{Biernacki:1936}, and the relationship between close-to-convexity and linear accessibility was established by Lewandowski \cite{Lewandowski:1958,Lewandowski:1960}.
Since these papers are written in French and the original proofs of the equivalence are rather involved, we recall below a short Loewner chain proof of the implication from close-to-convexity to linear accessibility due to Bielecki and Lewandowski \cite{BielLewa:1962}. For the converse implication, we refer to Koepf's simplified proof in English \cite{Koepf:1989}.

We use the following stability property later. 

%We shall use the following stability property later. Since the notion of linear accessibility goes back to Biernacki \cite{Biernacki:1936}, whose paper is written in French, we include the statement and proof for the reader's convenience.

\begin{lem}[{\cite[Lemme III, p.297]{Biernacki:1936}}]
\label{lemma-closedness-LA}
Let $\{f_n\}$ be a sequence of functions in $\mathcal S$ which converges locally uniformly in $\D$ to a function $f\in\mathcal S$. Suppose that each domain $f_n(\D)$ is linearly accessible. Then $f(\D)$ is also linearly accessible.
\end{lem}

%For the proof, see Appendix B.

We discuss Biernacki's original proof in Appendix~B.

\subsection{Loewner chains}

A family $(f_t)_{t\geq 0}$ of holomorphic functions on $\D$ is said to be a \textit{Loewner chain} if $a(t) : = f_{t}'(0)$ is locally absolutely continuous on $[0,\infty)$ with $d|a(t)|/d t  > 0$ for a.e. $t \ge 0$ and $|a(t)| \to \infty$ as $t \to \infty$, $f_{t}/a(t) \in \S$ for all $t \in [0,\infty)$ and $f_{s} (\D) \subsetneq f_{t} (\D)$ for all $s < t$. In this case, the strict inclusion implies that $|a(t)|$ is strictly increasing (see e.g. \cite{Hotta:2010a}).
It is well known that every Loewner chain $(f_{t})$ satisfies the Loewner differential equation
\begin{equation}
\label{LoewnerPDE}
\dot{f}_{t}(z)  =
f_{t}'(z)\cdot z p(z,t) 
\qquad
(\dot{f}_{t} := \partial f_{t}/\partial t, f_{t}' := \partial f_{t}/\partial z)
\end{equation}
for almost every $t\geq 0$, where $p(\cdot,t)$ is holomorphic in $\D$, $p(z,\cdot)$ is measurable, and $p(\D,t) \subset \{w \in \C : \Re w > 0\}$ for almost every $t\geq0$. 
%Conversely, such a $p$ with $\int_0^\infty \Re p(0,t)\,dt=\infty$ generates a locally absolutely continuous Loewner chain satisfying \eqref{LoewnerPDE}. 
We refer to \cite[Chapter~6]{Pom:1975} for the basic theory of Loewner chains.

By definition, $(f_{t})$ describes a growing family of simply connected domains $f_s(\D)\subsetneq f_t(\D)\,\, (s<t).$
From a different point of view, for each fixed $z\in\D$, the mapping
$t \mapsto f_t(z)$ defines a curve in the complex plane. 
Thus, a Loewner chain may also be regarded as a family of trajectories
$$
\{t \mapsto f_t(z): z \in \D\}
$$
associated with $(f_{t})$.

This geometric interpretation can be applied to close-to-convex functions as follows.
Let $f \in \CTC$, and choose a starlike function $g \in \S^{*}$ satisfying \eqref{definition-ctc}.
Then $f_t(z):=f(z)+t g(z)$ forms a Loewner chain (see e.g. \cite[p.12]{MR4018207}).
To see the geometry of this chain more precisely, fix $r \in (0,1)$ and put $f^{r}(z) := f(r z)$ and $g^{r}(z) := g(r z)$.
Then $f_t^{r} := f^{r} + t g^{r}$ is again a Loewner chain.
Moreover, $f_t^r$ is holomorphic in $\closure{\D}$. 
For each $\zeta\in\de\D$, the trajectory
$$
\gamma_\zeta^{r}
:=
\{f_t^r(\zeta):t\geq0\}
=
\{f^{r}(\zeta)+t g^{r}(\zeta):t\geq0\}
$$
is therefore a closed half-line issuing from $f^{r}(\zeta)$ with inclination $\arg g^{r}(\zeta)$.

By the univalence and nesting property of a Loewner chain, these trajectories do not cross each other. 
Further, since $\arg g^{r}(\zeta)$ is continuous and strictly increasing, one can show that $\bigcup_{\zeta \in \de\D} \gamma_{\zeta}^{r} = \C\,\backslash\,f^{r}(\D)$, which proves that $f^r(\D)=f(r\D)$ is linearly accessible for every $r\in(0,1)$.
Since $f^r \to f$ locally uniformly as $r\to1$,
Lemma~\ref{lemma-closedness-LA} implies that $f$ is linearly accessible.
This is essentially the short Loewner-chain proof of the implication from close-to-convexity to linear accessibility given by Bielecki and Lewandowski \cite{BielLewa:1962}.

Let us next apply this interpretation to the Noshiro--Warschawski class $\R$ (see also \cite{HottaWang:2014}). Suppose that $f\in\R$. Since the starlike function in \eqref{definition-ctc} can be chosen as $g(z)=z$, the preceding construction shows that
\begin{equation}
\label{eq:R-dilated-ray}
\{f(r\zeta)+t\zeta:t\geq0\}
\subset\C\setminus f(r\D),
\qquad r\in(0,1),\quad \zeta\in\de\D,
\end{equation}
after reparametrizing the half-line.
Suppose that the finite radial limit $f(\zeta)=\lim_{r\uparrow1}f(r\zeta)$ exists at $\zeta\in\de\D$.
Then
\begin{equation}
\label{eq:R-boundary-ray}
\gamma_\zeta
:=
\{f(\zeta)+t\zeta:t\geq0\}
\subset\C\setminus f(\D).
\end{equation}
Indeed, suppose to the contrary that
$w:=f(\zeta)+t\zeta\in f(\D)$ for some $t\geq0$.
Then there exists $r_{0} \in (0,1)$ such that $w\in f(r\D)$ for all $r \in [r_{0},1]$.
Since $f(r_0\D)$ is open and $f(r\zeta)+t\zeta\to w$ as $r\uparrow1$, we have
\[
f(r\zeta)+t\zeta\in f(r_0\D)\subset f(r\D)
\]
for all $r\in(r_0,1)$ sufficiently close to $1$, contrary to \eqref{eq:R-dilated-ray}.
%Since $f(r_0\D)$ is open and $f(r\zeta)+t\zeta\to w$ as $r \to 1$, we would have
%\[
%f(r\zeta)+t\zeta\in f(r_0\D)\subset f(r\D)
%\]
%for all $r \in [r_{0}, 1]$,
%contrary to \eqref{eq:R-dilated-ray}.
Thus, if $f(e^{i\theta})$ exists and is finite, then a closed half-line of direction angle $\theta$ issues from $f(e^{i\theta})$ and lies in $\C\setminus f(\D)$.
This observation gives some simple geometric obstructions for a function in $\R$.

Here, we say that a domain $D\subset\C$ has a \textit{slit}
if there exists a nondegenerate closed Jordan arc
$J\subset\partial D$ such that every point of $J$,
except for one of its endpoints, belongs to
$\operatorname{int}(D\cup J)$.

\begin{prop}
\label{prop:R-no-slit}
Let $f\in\S$. If $f(\D)$ has a slit, then $f\notin\R$.
\end{prop}

\begin{proof}
Suppose to the contrary that $f\in\R$.
Let $J\subset\partial f(\D)$ be a slit.
Since $f(\D)$ is linearly accessible, a closed half-line
contained in $\C\setminus f(\D)$ issues from any interior
point of $J$.
Its initial segment lies in $J$, so $J$ contains a
nondegenerate straight subarc.
We may therefore choose an interior point $\omega$ of
this subarc and an open disk $B$ centered at $\omega$
such that
$
B\setminus f(\D)=B\cap L,
$
where $L$ is the straight line containing the subarc.

By the local boundary correspondence on the two sides
of a straight slit, there exist distinct points
$\zeta_1,\zeta_2\in\de\D$ such that $f(\zeta_1)=f(\zeta_2)=\omega$.
By \eqref{eq:R-boundary-ray},
\[
\{\omega+t\zeta_j:t\geq0\}
\subset\C\setminus f(\D),
\qquad j=1,2.
\]
The initial portions of these half-lines must lie in $L$.
Since $\zeta_1\neq\zeta_2$, their directions are opposite,
and hence $\zeta_2=-\zeta_1$.
It follows that
\[
L
=
\{\omega+t\zeta_1:t\geq0\}
\cup
\{\omega+t\zeta_2:t\geq0\}
\subset\C\setminus f(\D).
\]
However, $B\setminus L\subset f(\D)$, so $f(\D)$ meets
both open half-planes bounded by $L$.
This contradicts the connectedness of $f(\D)$.
\end{proof}

\begin{prop}
\label{prop:R-no-angular-domain}
Let $f\in\S$.
If $f(\D)$ contains an angular domain, then $f\notin\R$.
\end{prop}

\begin{proof}
Suppose to the contrary that $f\in\R$ and $f(\D)$ contains a sector
\[
A=\{a+ w: \theta_1< \arg w<\theta_2\},
\]
where $a\in\C$ and $0<\theta_2-\theta_1<2\pi$.
Since every univalent function in $\D$ has a finite
radial limit almost everywhere on $\de\D$ (see e.g. \cite[p.11]{Pom:1992boundary}), we can choose
$\theta\in(\theta_1,\theta_2)$ such that
$f(\zeta)$ exists and is finite, where $\zeta=e^{i\theta}$.
By \eqref{eq:R-boundary-ray},
\[
\{f(\zeta)+t\zeta:t\geq0\}
\subset\C\setminus f(\D).
\]
On the other hand, since $\theta$ lies strictly between
the boundary directions of $A$,
\[
f(\zeta)+t\zeta\in A
\]
for all sufficiently large $t$.
This contradicts $A\subset f(\D)$.
\end{proof}

	%	++++++++++++++++++++++++++++++++++++++++++++++++++++++
	%
	%		Section 3
			\section{Boundary geometry of the Noshiro--Warschawski class}
	%
	%	++++++++++++++++++++++++++++++++++++++++++++++++++++++

\subsection{Boundary geometry and spherical length}

The discussion in the preceding section and the two propositions suggest that the boundary geometry of the image domain of a function in $\R$ is subject to rather strong restrictions. These observations suggest, at least heuristically, that the boundary of $f(\D)$ should be reasonably well-behaved. This raises the following question:

\begin{question}
\label{our-question}
Is the boundary of $f(\D)$ locally connected for every $f\in\R$?
\end{question}

In order to approach this problem, we consider the spherical length of the images of concentric circles under $f$.
Indeed, for a conformal map $f$ on $\D$, if 
$$
\Lambda_{\textup{sph}}(f)
:=
\sup_{0<r<1}
\int_{\partial\D_r}
\frac{|f'(z)|}{1+|f(z)|^2}\,|dz|
<\infty,
$$
where $\D_{r} :=\{z \in \C : |z| < r\}$, then $\partial f (\D)$ is locally connected. 
This is proved in Appendix~A; see Theorem~\ref{loca-connect-sph}.
Thus, one possible approach to Question \ref{our-question} would be to establish a uniform bound for the spherical lengths
for functions $f\in\R$. Unfortunately, this approach does not work in general.

\begin{thm}
For every $C>0$, there exists $f\in\R$ such that
$$
\Lambda_{\textup{sph}}(f)>C.
$$
\end{thm}

\begin{proof}
For a positive integer $n$, let $f_n(z)$ be analytic in $\D$ with $f_n(0)=0$ and
$$
f_n'(z)
=\frac{1}{n}\sum_{k=0}^{n-1}\frac{\omega_n^k+z}{\omega_n^k-z},\quad \textup{where}\quad \omega_n=e^{\frac{2\pi i}n}.
$$
Since $\Re f_{n}'(z) >0$ for all $z \in \D$, $f_{n} \in \R$.
Integrating the above expression gives
\begin{align*}
f_n(z)&=-\frac{2}{n}\sum_{k=0}^{n-1}\Biggl[\omega_n^k\log(1-\overline{\omega_n^k}z)\Biggr]-z\\
&=\frac{2}{n}\sum_{k=0}^{n-1}\Biggl[\,\sum_{m=1}^{\infty}\left(\overline{\omega_n^{k(m-1)}}\frac{z^m}{m}\right)\Biggr]-z\\
&=\frac{2}{n}\sum_{m=1}^{\infty}\left(\frac{z^m}{m}\sum_{k=0}^{n-1}\overline{\omega_n^{k(m-1)}}\right)-z\\
&=2\sum_{p=0}^{\infty}\frac{z^{np+1}}{np+1}-z.
\end{align*}
For $k=0,\ldots,n-1$, let $\xi_k:=e^{\frac{(2k+1)\pi i}n}$.
Then $(\xi_k)^n=-1$, and hence
$$
f_n(\xi_k)
=
\xi_k
\left(
2\sum_{p=0}^{\infty}
\frac{(-1)^p}{np+1}-1
\right).
$$
Since $\frac{n}{n+1}<\sum_{p=0}^{\infty}\frac{(-1)^p}{np+1}<1$, we have $|f_n(\xi_k)|<1$.
On the other hand, $f_n(r\omega_n^k)\to\infty$ as $r\to1^-$.

For $0<r<1$, let $I_{k,r}$ denote the circular arc of
$\partial\D_r$ joining $r\omega_n^k$ to $r\xi_k$ in the
counterclockwise direction. The arcs $I_{k,r}$,
$k=0,\ldots,n-1$, are mutually disjoint. Since the
length of a curve is not smaller than the distance between its
endpoints,
$$
\int_{I_{k,r}}
\frac{|f_n'(z)|}{1+|f_n(z)|^2}\,|dz|
\geq
s\bigl(f_n(r\omega_n^k),f_n(r\xi_k)\bigr).
$$
Consequently,
$$
\int_{\partial\D_r}
\frac{|f_n'(z)|}{1+|f_n(z)|^2}\,|dz|
\geq
\sum_{k=0}^{n-1}
s\bigl(f_n(r\omega_n^k),f_n(r\xi_k)\bigr).
$$
Letting $r\to1^-$, we have $f_n(r\omega_n^k)\to\infty$ and $f_n(r\xi_k)\to f_n(\xi_k)$.
Since $|f_n(\xi_k)|<1$, we have
$$
s(\infty,f_n(\xi_k))=
\int_{|f_n(\xi_k)|}^{\infty}\frac{dt}{1+t^2}
=
\frac{\pi}{2}-\arctan |f_n(\xi_k)|
>
\frac{\pi}{4}.
$$
It follows that
$$
\liminf_{r\to1^-}
\int_{\partial\D_r}
\frac{|f_n'(z)|}{1+|f_n(z)|^2}\,|dz|
\geq
\sum_{k=0}^{n-1}
s(\infty,f_n(\xi_k))
>
\frac{n\pi}{4}.
$$
Therefore $\Lambda_{\text{sph}}(f_n)>n\pi/4$. Choosing $n$ sufficiently large so that $n\pi/4>C$ proves the
assertion.
\end{proof}

\begin{rem}
Even when $\partial f(\D)$ is rectifiable, the spherical length of
$\partial f(\D)$ need not coincide with $\Lambda_{\text{sph}}(f)$.
For example, consider $f(z)=z/(1-z)$, which belongs to $\S$. A direct computation shows that
$$
\int_{\partial\D_r}
\frac{|f'(z)|}{1+|f(z)|^2}\,|dz|
=
\frac{2\pi r}{\sqrt{1+4r^4}}.
$$
Hence
$\Lambda_{\text{sph}}(f)=\pi$, whereas
$$
\ell_s(\partial f(\D)) = \lim_{r\to1^-}
\int_{\partial\D_r}
\frac{|f'(z)|}{1+|f(z)|^2}\,|dz|
=
\frac{2\pi}{\sqrt5}.
$$
Thus, in general, the supremum over $0<r<1$ is strictly larger than
the spherical length of the boundary.
\end{rem}

\subsection{A negative answer to the question of local connectedness}

While exploring several possible approaches to Question~\ref{our-question},
we found a function in $\R$ whose image domain has a non-locally connected
boundary. This gives a negative answer to the question.

\begin{thm}
\label{negative-R-thm}
There exists $f \in \R$ such that $\de f(\D)$ is not locally connected.
\end{thm}

\begin{proof}

Let $\omega_{n} := e^{i\pi/2^{n}}$ and
\begin{align*}
\phi(z) 
	& := \sum_{n=1}^{\infty} \frac{1}{2^{n}}\bigg(-z-2\omega_n \log(1-\bar{\omega}_{n}z)\bigg)\\
	& = -z -2\sum_{n=1}^{\infty} \frac{1}{2^{n}}\bigg(\omega_n \log(1-\bar{\omega}_{n}z)\bigg),
\end{align*}
where the branch of $\log$ is taken so that $\log 1=0$.
Note that $\phi\in\R$.
This function will provide the desired counterexample. We shall show that $\de\phi(\D)$ is not locally connected.
We first establish the following.

\begin{claim}
There exists an $M>0$ such that $|\phi(e^{i\pi\frac{3}{2^{k+2}}})| < M$ for any $k \in \N$.
\end{claim}

\begin{proof}[\textbf{Proof of the claim}]
Let $\xi_k:=e^{i\pi\frac{3}{2^{k+2}}}$.
For our purpose, it suffices to show that
\begin{align}
\label{claim-003}
\sum_{n=1}^{\infty}\frac1{2^n}\Bigl|\log(1-\bar\omega_n\xi_k)\Bigr|
\leq
\underbrace{\sum_{n=1}^{\infty}\frac1{2^n}
\Bigl|
\log|1-\bar\omega_n\xi_k|
\Bigr|}_{=:A}
+
\underbrace{\sum_{n=1}^{\infty}\frac1{2^n}
\Bigl|
\arg(1-\bar\omega_n\xi_k)
\Bigr|}_{=:B}
\end{align}
is uniformly bounded in $k$.
As for the term $B$ in \eqref{claim-003}, since the branch of $\log$ is taken as $\log 1=0$, we have
$\left|\arg(1-\bar\omega_n\xi_k)\right|<\pi/2$ and hence $B < \pi /2$.
Therefore, it remains to estimate the term $A$.

Since $1 - e^{i\theta} = -2i e^{(i\theta/2)} \sin (\theta/2)$, setting $\delta_{n,k}
:=\left|\frac1{2^n}-\frac3{2^{k+2}}\right| \in (0,1)$ we have
$$
|1-\bar\omega_n\xi_k|
=
2\sin\left(\frac{\pi}{2}\delta_{n,k}\right)
\geq
2\delta_{n,k}.
$$
Here we have used the elementary estimate $\sin x \geq \frac{2}{\pi}x$ for $0\leq x\leq\frac{\pi}{2}$. 
On the other hand, $|1-\bar\omega_n\xi_k|\leq2$. 
Hence
$$
\delta_{n,k}
\leq
\frac{|1-\bar\omega_n\xi_k|}{2}
\leq1.
$$
It follows that
$$
\left|
\log\frac{|1-\bar\omega_n\xi_k|}{2}
\right|
\leq
-\log\delta_{n,k}
$$
and therefore
\begin{align*}
\Bigl|\log |1-\bar\omega_n\xi_k|\Bigr|
&\leq
\log2+
\left|
\log\frac{|1-\bar\omega_n\xi_k|}{2}
\right|\\
&\leq
\log2-\log\delta_{n,k}.
\end{align*}
In order to estimate the term $- \log \delta_{n,k}$, we distinguish two cases. 
If $1\leq n\leq k$, then $\delta_{n,k}= (1-3\cdot2^{n-k-2})/2^n \geq 1/2^{n+2}$, and hence
\begin{equation}
\label{claim-001}
-\log\delta_{n,k}
\leq
(n+2)\log2.
\end{equation}
If $n\geq k+1$, then
$\delta_{n,k}=(3-2^{k+2-n})/2^{k+2}\geq 1/2^{k+2}$, so that
\begin{equation}
\label{claim-002}
-\log\delta_{n,k}
\leq
(k+2)\log2.
\end{equation}
Combining \eqref{claim-001} and \eqref{claim-002}, we obtain
\begin{align*}
-\sum_{n=1}^{\infty}
\frac1{2^n}\log\delta_{n,k}
&\leq
\log2 \cdot
\left(
\sum_{n=1}^{k}\frac{n+2}{2^n}
+
(k+2)\sum_{n=k+1}^{\infty}\frac1{2^n}
\right)\\
&\leq
\log2 \cdot
\left(
4+\frac{k+2}{2^k}
\right)
 \leq \frac{11}2 \log 2.
\end{align*}
Consequently $A \leq \frac{13}{2}\log2$.
Thus both $A$ and $B$ are uniformly bounded in $k$, and hence so is the left-hand side of
\eqref{claim-003} as well.

Since $|\omega_n|=|\xi_k|=1$, it follows that
\begin{align*}
|\phi(\xi_k)|
&\leq
1+
2\sum_{n=1}^{\infty}\frac1{2^n}
\Bigl|\log(1-\bar\omega_n\xi_k)\Bigr| < 1+13\log2+\pi
\end{align*}
and hence $|\phi(\xi_k)|$ is uniformly bounded in $k$.
The proof of our claim is complete.
\renewcommand{\qedsymbol}{$\blacksquare$}
\end{proof}

We now go back to the proof of Theorem \ref{negative-R-thm}. 
Suppose to the contrary that $\de\phi(\D)$ is locally connected. 
This is equivalent to saying that $\phi$ admits a continuous extension
$\Phi:\closure{\D}\to \CC$.
Consider the sequence $\{\omega_k\}=\{e^{i\pi/2^k}\}$ which converges to 1. Then $\Phi(\omega_k) = \infty$ for all $k \in \N$.
Since $\Phi$ is continuous on $\closure{\D}$, we have
\begin{equation}
\label{loc-connected-01}
\Phi(1) = \lim_{k \to \infty}\Phi(\omega_k) =\infty.
\end{equation}

On the other hand, consider $\{\xi_k\}=\{e^{i\pi\frac{3}{2^{k+2}}}\}$ which also converges to 1. 
The preceding claim says that there exists a constant $M$ such that
$|\Phi(\xi_k)|<M$ for every $k\in\N$. Again, by the continuity of $\Phi$, we have
\begin{equation*}
|\Phi(1)| = \lim_{k \to \infty}|\Phi(\xi_k)| \leq M.
\end{equation*}
This contradicts \eqref{loc-connected-01}.
We conclude that $\de\phi(\D)$ is not locally connected.
\end{proof}

\section{A simple example in $\R\setminus\S^{*}$}

It is known that neither of the classes $\R$ and $\S^{*}$ is contained
in the other. Indeed, the Koebe function belongs to $\S^{*}$ but not to
$\R$. On the other hand, Krzy\.{z} \cite{Krzyz:1962} first showed that
the condition $\Re f'(z)>0$ does not imply starlikeness.
A somewhat simplified version of his counterexample was later given in
\cite[Example~4.6]{KimPonnu:2004a}:
$$
f(z)=\int_0^z
\left\{
(1+i)\sqrt{\frac{\alpha-u}{\alpha+u}}-i
\right\}\,du,
\qquad
\alpha=\frac{3+4i}{5}.
$$
This function belongs to $\R$ but not to $\S^{*}$.
See also Mocanu \cite{MR948445} for another example and related
sufficient conditions for starlikeness.

We present another example that seems to be even simpler.
Let $R(z) :=-z-2\log(1-z)$ and its rotation $R_{t}(z) := e^{it}R(e^{-it}z)$ with angle $t \in [0,2\pi)$.

\begin{thm}
For every $0<s<\log 2/(\pi+\log 2) \approx 0.18075$,
the function
$$
F_s(z):=sR(z)+(1-s)R_{\pi/2}(z)
$$
belongs to $\R\setminus\S^{*}$.
\end{thm}

\begin{proof}
It is clear that $F_s \in \R$, so we will show that $F_s\notin\S^{*}$ whenever $0<s<\log 2/(\pi+\log 2)$.
Let $z=e^{it}$ with $t>0$ sufficiently small. 
We would like to observe the sign of
$$
\Re \frac{e^{it}F_s'(e^{it})}{F_s(e^{it})}
=
\frac{
\Re[\,e^{it}F_s'(e^{it})\overline{F_s(e^{it})}\,]
}{
|F_s(e^{it})|^2
}.
$$
Further, multiplying by $t >0$ does not change the sign.
Hence it suffices to show that 
$$
t \cdot\Re[e^{it}F_s'(e^{it})\overline{F_s(e^{it})}]
=
\Re[e^{it}F_s'(e^{it})] \cdot \Re[t F_s(e^{it})]
+
\Im [t e^{it}F_s'(e^{it})] \cdot \Im[F_s(e^{it})]
$$
is negative for some $t$. A direct calculation gives
$$
\left\{
\begin{array}{ll}
%\dstyle\lim_{t\to0+}t\,\Re F_s(e^{it}) = \lim_{t\to0+}t(-\cos t-\log(2-2\cos t)) = 0,
\dstyle\lim_{t\to0+}t\,\Re F_s(e^{it})
=
\lim_{t\to0+}
t\left(
-\cos t
-s\log(2-2\cos t)+(1-s)\left(\frac{\pi}{2}+t\right)
\right)
=0.\\[7pt]
\dstyle\lim_{t\to0+}\Im F_s(e^{it}) = \lim_{t\to0+}\bigl(-\sin t-st+s\pi -(1-s)\log(2-2\sin t)\bigr) = s\pi-(1-s)\log2,\\[7pt]
\dstyle\lim_{t\to0+}\Re\left(e^{it}F_s'(e^{it})\right)
=-2s,\\[7pt]
\dstyle\lim_{t\to0+}
t\Im\left(e^{it}F_s'(e^{it})\right)
=2s.
\end{array}
\right.
$$
Therefore
$$
\lim_{t\to0+}
t\Re\left(
e^{it}F_s'(e^{it})\overline{F_s(e^{it})}
\right)
=
%2s\{s\pi-(1-s)\log2\}
2s\{s(\pi+\log 2)-\log2\}<0,
$$
where the last inequality follows from
$
0<s<\log2/(\pi+\log2).
$
We conclude that $F_s$ is not starlike.
\end{proof}

Figure~\ref{nonstarlike} shows the boundary curve of $F_s(\D)$ for
$s=0.1$. 
One can see that the image domain is not starlike with respect to the origin.
\begin{figure}[h]
\centering
\includegraphics[width=220pt]{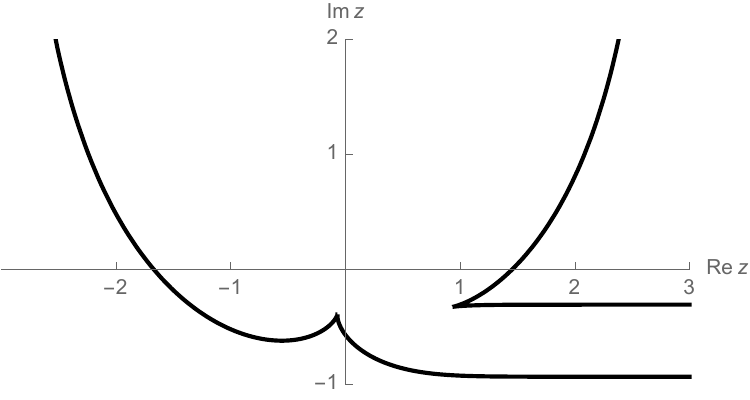}
\vspace{-5pt}
\caption{The shape of $\partial F_{s}(\D)$ when $s=0.1$.}
\label{nonstarlike}
\end{figure}

%\begin{rem}
%For every $s\in(0,1)$ and $u\in(0,2\pi)$,
%$F_{s,u}\in\S^*$ if and only if $F_{1-s,u}\in\S^*$.
%Indeed, this follows from rotation and reflection, since $
%e^{-iu}F_{s,u}(e^{iu}z)=F_{1-s,2\pi-u}(z)
%$
%and
%$
%\overline{F_{s,u}(\overline{z})}=F_{s,2\pi-u}(z).
%$
%Consequently, $F_s$ also does not belong to $\S^*$ if
%$\pi/(\pi+\log 2)<s<1$.
%\end{rem}

\appendix

\section{Spherical distance and local connectedness}

Let $s$ denote the spherical distance on $\CC$ induced by
$$
ds=\frac{|dw|}{1+|w|^{2}}.
$$
We denote by $\mathcal H_s^{1}$ the corresponding one-dimensional
Hausdorff measure. More generally, for a metric space $(X,\rho)$ and
a set $E\subset X$, the one-dimensional Hausdorff measure of $E$
with respect to $\rho$ is defined by
$$
\mathcal H_{\rho}^{1}(E)
:=
\lim_{\delta\downarrow0}\,
\inf
\left\{
\sum_{j=1}^{\infty}\operatorname{diam}_{\rho}(U_j)
:
E\subset\bigcup_{j=1}^{\infty}U_j,\
\operatorname{diam}_{\rho}(U_j)<\delta
\right\}.
$$
For general facts about Hausdorff measures, see e.g. \cite{MR4520153}.

We shall use two classical results concerning continua of finite
one-dimensional Hausdorff measure. The first one provides a
parametrization of such continua and, in particular, allows us to
deduce local connectivity.

\begin{lem}
\label{Wazewski-lemma}
Let $(X,\rho)$ be a metric space and let $K\subset X$ be a continuum.
If $\mathcal H_{\rho}^{1}(K)<\infty$, then there exists a Lipschitz mapping $L:[0,1]\to X$ such that $L([0,1])=K$.
\end{lem}

The existence of such a Lipschitz parametrization goes back to
Wa{\.z}ewski \cite{Wazewski1927}. See \cite[Theorem 2.1]{MR4520153} for a concise formulation, and \cite[Theorem 4.4]{AlbertiOttolini2017} for a more detailed treatment in the setting of metric spaces; see also the references therein.

The second fact is the lower semicontinuity of the one-dimensional
Hausdorff measure under Hausdorff convergence of continua, usually
referred to as \textit{Go{\l}\k{a}b's semicontinuity theorem}. 
Recall that for two nonempty compact subsets $A,B \subset X$, their Hausdorff distance is defined by
$$
d_{\mathcal H}^{\rho}(A,B)
:=
\max\left\{
\sup_{a\in A}\rho(a,B), \,\sup_{b\in B}\rho(b,A)
\right\}.
%,\quad\textup{where}\quad\rho(x,A):=\inf_{a\in A}\rho(x,a).
$$
\begin{lem}
\label{Golab-lemma}
Let $(X,\rho)$ be a metric space, and let ${K_n}$ be a sequence of
continua in $X$. Suppose that $K_n$ converges in the Hausdorff
distance to a continuum $K\subset X$, that is,
$d_{\mathcal{H}}^{\rho}(K_n,K) \to 0$ as $n \to \infty$.
Then
\[
\mathcal H_{\rho}^{1}(K) \leq \liminf_{n\to\infty} \mathcal{H}_{\rho}^{1}(K_n).
\]
\end{lem}

This property goes back to a work of Go{\l}\k{a}b \cite{Golab1929}. For the Euclidean formulation, see
\cite[Theorem 2.2]{MR4520153}. A formulation and proof in the
more general setting of metric spaces can be found in \cite[Theorem 2.9]{AlbertiOttolini2017}.

We now apply these two results to the boundary behavior of conformal mappings.

\begin{thm}
\label{loca-connect-sph}
Let $f$ be conformal of $\D$ onto a domain
$\Omega:=f(\D)$. 
Suppose that
\[
\Lambda_{\textup{sph}}(f) =
  \sup_{0<r<1}
  \int_{\partial\D_r}
  \frac{|f'(z)|}{1+|f(z)|^2}\,|dz|
  <\infty,
\]
where $\D_r=\{z:|z|<r\}$. Then $\partial\Omega$ is locally connected.
\end{thm}

\begin{proof}

Let $\mathcal{H}_{s}^{1}$ be the one-dimensional Hausdorff measure associated with the spherical distance, and $\Omega_{r} := f(\D_{r})$.
Recall that the spherical length of the curve $\partial \Omega_{r}$ is given by
\[
\ell_{s}(\partial \Omega_{r})
 =
\int_0^{2\pi}
  \frac{r|f'(re^{i\theta})|}
       {1+|f(re^{i\theta})|^2}\,d\theta.
\]
In general, the one-dimensional Hausdorff measure of the image of a rectifiable curve in a metric space does not exceed its length. Hence we have
\begin{equation}
\label{localthmintro-eq01}
\mathcal{H}_{s}^{1}(\partial \Omega_{r}) \leq \ell_{s}(\partial \Omega_{r}).
\end{equation}
Since $f$ is conformal on $\D$, its restriction to $\partial\D_r$ is injective, and hence $\partial \Omega_{r}$ is a simple closed curve. 
Therefore, equality actually holds above, although the inequality is sufficient for our purpose.
It follows from our assumption and inequality \eqref{localthmintro-eq01} that $\mathcal{H}_{s}^{1}(\partial \Omega_{r})$ has a uniform bound with respect to $r \in (0,1)$.

Next, in order to apply Lemma \ref{Golab-lemma}, we show that
$$
d_{\mathcal H}^{s}(\partial\Omega_r,\partial\Omega)\to0
\quad\textup{as}\quad r\to1.
$$
By the definition of the Hausdorff distance, it suffices to prove that
\begin{equation}
\label{localthm-inordertoshow}
\sup_{w\in\partial\Omega_r}s(w,\partial\Omega)\to0
\qquad\textup{and}\qquad
\sup_{\zeta\in\partial\Omega}s(\zeta,\partial\Omega_r)\to0
\end{equation}
as $r\to1$.

To prove the first convergence, suppose to the contrary that there exist a sequence $\{r_n\}$ with $r_n\uparrow1$ and a constant $\varepsilon>0$ such that, for every $n\in\N$ there exists $w_n\in\partial\Omega_{r_n}$ satisfying
\begin{equation}
\label{localthm-eq01}
s(w_n,\partial\Omega)\geq\varepsilon.
\end{equation}
By the compactness of $\CC$, passing to a subsequence if necessary, we may assume that $w_n\to w$ in the spherical metric.
As $\partial\Omega_{r_n}\subset\Omega$, we have $w\in\overline{\Omega}=\Omega\cup\partial\Omega$.
\eqref{localthm-eq01} implies that $w\notin\partial\Omega$, so we have $w\in\Omega$.
Choose a compact neighborhood $K$ of $w$ such that $K\subset\Omega$ and fix it. 
Since $\{w_{n}\}$ converges to $w$, for all sufficiently large $n$, we have $w_n\in K$.
On the other hand, the domains $\Omega_r$ form an increasing exhaustion of $\Omega$, namely
\begin{equation}
\label{locallem-exhaustion}
\Omega_{s}\subset\Omega_{r} \quad (s<r),
\qquad
\Omega=\bigcup_{0<r<1}\Omega_r.
\end{equation}
Hence, for all sufficiently large $m$, we have $K\subset\Omega_{r_{m}}$
Since $w_{m} \in K$ and $w_{m} \in \partial \Omega_{r_m}$, this implies
\[
w_m\in\Omega_{r_m}\cap\partial\Omega_{r_m} = \emptyset,
\]
a contradiction.
We conclude that the first convergence in \eqref{localthm-inordertoshow} holds.

As for the second convergence, suppose again to the contrary that there exist a sequence $\{r_n\}$ with $r_n\uparrow1$ and a constant $\varepsilon>0$ such that, for every $n\in\N$ there exists $\zeta_n\in\partial\Omega$ satisfying
\begin{equation}
\label{localthm-eq02}
s(\zeta_{n},\partial\Omega_{r_{n}})\geq\varepsilon.
\end{equation}
Since $\partial\Omega$ is compact in the spherical metric, passing to a subsequence, we may assume that $\zeta_{n} \to \zeta \in \partial \Omega$.
Hence for all sufficiently large $n$, $s(\zeta, \zeta_{n}) < \varepsilon/2$, and therefore $B_s\left(\zeta,\varepsilon/2\right) \subset B_s(\zeta_n,\varepsilon)$, where $B_{s}(\zeta, \varepsilon) := \{w \in \CC : s(\zeta, w) < \varepsilon\}$.
By \eqref{localthm-eq02}, $B_{s}(\zeta_{n}, \varepsilon)$ and $\Omega_{r_{n}}$ are disjoint. Consequently, we have
$$
B_s\left(\zeta,\frac{\varepsilon}{2}\right)
\cap\Omega_{r_n}=\emptyset.
$$
On the other hand, since $\zeta$ is a boundary point of $\Omega$,
$\Omega \cap B_{s}(\zeta, \varepsilon/2)$ is nonempty. 
Thus, there exists a point $w \in \Omega \cap B_{s}(\zeta, \varepsilon/2)$ which is not contained in any $\Omega_{r_{n}}$.
This fact contradicts the exhaustion property \eqref{locallem-exhaustion}.

Now we have proved that $d_{\mathcal H}^{s}(\partial\Omega_{r_n},\partial\Omega)\to 0$.
Therefore, by Lemma \ref{Golab-lemma},
$$
\mathcal H_s^1(\partial\Omega)
\leq
\liminf_{n\to\infty}
\mathcal H_s^1(\partial\Omega_{r_n})
\leq \Lambda_{\textup{sph}}(f).
$$
In particular $\mathcal H_s^1(\partial\Omega)<\infty$.
Since $\partial\Omega$ is a continuum, by Lemma \ref{Wazewski-lemma}, there exists a Lipschitz surjection
$
L:[0,1]\to\partial\Omega.
$
Hence $\partial\Omega$ is a curve in the sense of
\cite[Theorem 2.1]{Pom:1992boundary}. By the equivalence of conditions
{\rm (ii)} and {\rm (iii)} in that theorem, $\partial\Omega$ is
locally connected.
\end{proof}

\begin{rem}
The converse of the above theorem does not hold in general.
For example, let $\Omega$ be the interior of the Koch snowflake and let $f:\D\to\Omega$ be a conformal mapping. Then
$\partial\Omega$ is a Jordan curve, and hence is locally connected, whereas $\mathcal{H}_s^1(\partial\Omega)=\infty$.
Since $\mathcal{H}_s^1(\partial\Omega) \leq \Lambda_{\textup{sph}}(f)$
the latter supremum is infinite.
\end{rem}

The conformality assumption in Theorem \ref{loca-connect-sph} can be relaxed if one
assumes that the boundary of the image domain is a continuum.
Indeed, let $f$ be a nonconstant holomorphic function in $\D$, and suppose that $\partial f(\D)$ is a continuum and $\Lambda_{\textup{sph}}(f)<\infty$.
Set
$$
K_r:=f(\partial\D_r).
$$
Remark that $\de \Omega_{r} = \partial f(\D_{r})  \subset f(\partial\D_r) = K_{r}$, and $\de \Omega_{r}$ may not be a continuum.
Then $K_r$ is a continuum and $\mathcal H_s^1(K_r) \le \Lambda_{\textup{sph}}(f)$.
Moreover, since $\Omega_r=f(\D_r)$ increases to $\Omega$ and
$\partial\Omega_r\subset K_r$, every Hausdorff limit $K$ of a
sequence $K_{r_n}$ with $r_{n}\uparrow1$ contains
$\partial\Omega$. Hence, by the lower semicontinuity of
$\mathcal H_s^1$,
$$
\mathcal H_s^1(\partial\Omega)
\le
\mathcal H_s^1(K)
\le
\liminf_{n\to\infty}\mathcal H_s^1(K_{r_n})
<\infty.
$$
By Wa\.zewski's theorem, there exists a Lipschitz surjection from $[0,1]$ onto $\partial\Omega$. 
Hence, the Hahn--Mazurkiewicz theorem (see e.g. \cite[Theorem~8.14, p.~126]{MR1192552}) implies that $\partial\Omega$ is locally connected.
Summarizing the argument, we obtain the following theorem.

\begin{thm}
Let $f$ be a nonconstant holomorphic function in $\D$.
Suppose that $\partial f(\D)$ is a continuum and
$$
\Lambda_{\textup{sph}}(f)<\infty.
$$
Then $\partial f(\D)$ is locally connected.
\end{thm}

\section{Local uniform convergence of linearly accessible functions}

Lemma~\ref{lemma-closedness-LA} in Section~2 asserts that the class of linearly accessible functions is closed in $\S$ with respect to locally uniform convergence. Since $\S$ is compact, this class is therefore compact as well. This closedness result goes back to Biernacki \cite[Lemme~III, p.~297]{Biernacki:1936}.
The original proof is written in French and uses the terminology and style of its time. For this reason, we revisit the argument here using modern terminology. Note that in this appendix, boundaries are taken in $\C$.

In this appendix, two half-lines are said to \textit{cross} if they intersect transversally at a point belonging to the relative interior of both half-lines. Under this definition, overlapping collinear half-lines are not regarded as crossing.

We first isolate the part of Biernacki's argument that can be obtained directly from the convergence of the image domains.

%
%
%\cite[Lemme III, p.~297]{Biernacki:1936}.
%Since the original proof is written in French and uses the terminology and style of its time, it may be less accessible to modern readers.
%For this reason, we revisit Biernacki's argument and try to present a proof in English using modern terminology. In doing so, we pay particular attention to the final geometric step, which is only briefly explained in the original paper.
%
%
%
%For this reason, we revisit the original argument using modern
%terminology. We give a detailed account of the construction of
%limiting half-lines and discuss the remaining geometric completion
%step separately.
%
%
%
%We first isolate the part of Biernacki's argument that can be obtained
%directly from the convergence of the image domains.

\begin{thm}
\label{LA-thm01}
Let $\{f_n\}$ be a sequence of functions in $\S$ converging locally
uniformly in $\D$ to a function $f\in\S$, and set
$\Omega_n:=f_n(\D)$ and $\Omega:=f(\D)$.
Suppose that each domain $\Omega_{n}$ is linearly accessible.
Then, for every finite boundary point $w\in\partial\Omega\cap\C$, there exists a closed half-line
$L_w$ issuing from $w$ such that
$$
L_w\subset\C\setminus\Omega.
$$
Moreover, the half-lines $L_w$ can be chosen without crossings.
\end{thm}

\begin{proof}
By the Carath\'eodory kernel convergence theorem, the domains $\Omega_n$ converge to $\Omega$ in the sense of kernel convergence with reference to the origin $0$.
Hence, by \cite[Theorem 6.2 (i)]{Yanagihara:2022}, for every finite point $\zeta\in\partial\Omega$, there exist points $\zeta_n\in\partial\Omega_n$ such that $\zeta_n\to\zeta$.
For every $n$, choose a representation of $\C\setminus\Omega_n$ as a union of closed half-lines satisfying the defining non-intersection condition for linear accessibility.
Thus, for every $\zeta_n\in\partial\Omega_n$, there exists a half-line of this representation containing $\zeta_n$. If necessary, replace it by the sub half-line beginning at $\zeta_n$. Then we obtain a closed half-line issuing from $\zeta_n$ and contained in $\C\setminus\Omega_n$.

We first carry out the construction on a countable dense subset
of $\partial\Omega$. Choose $E=\{w_1,w_2,\ldots\}\subset\partial\Omega$ dense. For each $j$, choose a sequence
$\{w_{j,n}\}_n$ with
$$
w_{j,n}\in\partial\Omega_n
\qquad\text{and}\qquad
w_{j,n}\to w_j,
$$
and choose a corresponding half-line
$$
L_{j,n}
:=
\{w_{j,n}+t e^{i\theta_{j,n}}:t\geq0\}
\subset\C\setminus\Omega_n.
$$
First consider the sequence $\{e^{i\theta_{1,n}}\}_{n\in\N}\subset\partial\D$. Since $\partial\D$ is compact, there exists an infinite index set $\Lambda_1\subset\N$ such that $e^{i\theta_{1,n}}$
converges as $n\to\infty$ with $n\in\Lambda_1$.
Next, consider the sequence $\{e^{i\theta_{2,n}}\}_{n\in\Lambda_1}\subset\partial\D$.
Again by compactness, there exists an infinite index subset
$\Lambda_2\subset\Lambda_1$ such that $e^{i\theta_{2,n}}$ converges as $n\to\infty$ with
$n\in\Lambda_2$.
Continuing inductively, we obtain a decreasing sequence of infinite
index sets
$
\Lambda_1\supset\Lambda_2\supset\cdots
$
such that, for each $j$, the sequence
$
\{e^{i\theta_{j,n}}\}_{n\in\Lambda_j}
$
converges.
By a diagonal argument, after passing to a common subsequence and
relabeling, we may therefore assume that
$$
e^{i\theta_{j,n}}\to e^{i\theta_j}\quad (n \to \infty)
$$
for every $j\in\N$.

For each $j$, define the limiting half-line
$$
L_{j}
:=
\{w_j+t e^{i\theta_j}:t\geq 0\}.
$$
We claim that $L_j\subset\C\setminus\Omega$.
Suppose to the contrary that $w_{j}+t_{0}e^{i\theta_j}\in\Omega$ for some $t_{0}\geq 0$. 
Since $\Omega$ is open, there exists
$\varepsilon>0$ such that
$
\{w \in \C : |w - (w_j+t_0e^{i\theta_j})| \leq \varepsilon\}
\subset\Omega.
$
By kernel convergence, this closed disk is contained in
$\Omega_n$ for all sufficiently large $n$. On the other hand,
$w_{j,n}+t_0e^{i\theta_{j,n}}$ converges to $w_j+t_0e^{i\theta_j}$, and hence
$w_{j,n}+t_0e^{i\theta_{j,n}}\in\Omega_n$ for all sufficiently large $n$. 
This contradicts $L_{j,n}\subset\C\setminus\Omega_n$.
Therefore $L_j\subset\C\setminus\Omega$.

The half-lines $\{L_j\}$ do not cross one another at any point other than their initial points. Indeed, if two of them crossed at a point distinct from both of their initial points, then the corresponding approximating half-lines $L_{j,n}$ would also cross for
all sufficiently large $n$, which is impossible.

We now extend the construction to every point of $\partial\Omega$. Let $w\in\partial\Omega$. Since $E$ is dense in $\partial\Omega$, we may choose a sequence $\{w_{j_k}\}\subset E$ such that
$w_{j_k}\to w$ as $k\to\infty$.
If $w\in E$, we may simply take the constant sequence $w_{j_k}=w$.
Passing to a subsequence if necessary, we may assume that
$
e^{i\theta_{j_k}}\to e^{i\theta}
$
for some $\theta\in[0,2\pi)$.
Define
$$
L_w
:=
\{w+t e^{i\theta}:t\geq0\}.
$$
We claim that $L_w\subset\C\setminus\Omega$.
Suppose to the contrary that $w+t_0e^{i\theta}\in\Omega$
for some $t_0\geq0$.
Since $w_{j_k}+t_0e^{i\theta_{j_k}} \to w+t_0e^{i\theta}$
and $\Omega$ is open, we have $w_{j_k}+t_0e^{i\theta_{j_k}}\in\Omega$ for all sufficiently large $k$.
This contradicts $w_{j_k}+t_0e^{i\theta_{j_k}}\in L_{j_k}
\subset\C\setminus\Omega$. Hence $L_w\subset\C\setminus\Omega$.

It remains to verify the noncrossing property.
Let $w,v\in\partial\Omega$ with $w\neq v$, and suppose to the contrary that $L_w$ and $L_v$ cross at a point distinct from their initial points.
Choose sequences
$
w_{j_k}\to w
$
and
$
w_{\ell_k}\to v
$
in $E$ such that
$
e^{i\theta_{j_k}}\to e^{i\theta_w}
$
and
$
e^{i\theta_{\ell_k}}\to e^{i\theta_v},
$
where $e^{i\theta_w}$ and $e^{i\theta_v}$ are the directions of
$L_w$ and $L_v$, respectively.
Since the crossing happens away from the starting points, they will still cross even if we make very small changes to the starting points and directions.
Hence, for all sufficiently large $k$, the half-lines $L_{j_k}$ and $L_{\ell_k}$ also cross.
This contradicts the noncrossing property of the family $\{L_j\}$.
Therefore the family $\{L_w:w\in\partial\Omega\}$ can be chosen without crossings.
This yields the desired result, and the proof is complete.
\end{proof}

\begin{rem}
The noncrossing property in Theorem~\ref{LA-thm01} is weaker than the
non-intersection condition in the definition of linear accessibility, because collinear overlaps are not excluded. For example, the family $L_1=[0,\infty), L_2=[1,\infty)$ and $L_3=(-\infty,2]$
satisfies the noncrossing condition in the sense defined above.
Here $L_2\subset L_1$, while $L_1$ and $L_3$ overlap collinearly and point in opposite directions.
It seems plausible that an opposite-direction overlap such as $L_1$ and $L_3$ cannot occur among the limiting half-lines constructed in Theorem~\ref{LA-thm01}. However, this does not follow from the limiting argument given in the proof of Theorem~\ref{LA-thm01} directly.
\end{rem}

It remains to show that the conclusion of Theorem~\ref{LA-thm01} implies that $\Omega$ is linearly accessible. The essential point is the following geometric assertion.
Let
$$
\mathcal L
:=
\bigcup_{w\in\partial\Omega}L_w,
$$
and let $Q$ be a connected component of $(\C\setminus\Omega)\setminus\mathcal L = \C \setminus (\mathcal L \cup \Omega)$.
Here, by an \textit{angular set} with vertex $a$, we mean a set of the form $\{a+\rho e^{i\theta}:\rho>0,\ \theta\in I\}$, where $\theta_1 \leq \theta_2<\theta_1+2\pi$ and $I$ is one of the intervals $(\theta_1,\theta_2), [\theta_1,\theta_2), (\theta_1,\theta_2], [\theta_1,\theta_2]$.
The point at issue is whether the following is valid.
\begin{problem}
\label{LA-problem}
$Q$ is necessarily an angular set whose vertex belongs to $\partial\Omega$.
\end{problem}
\noindent
A short argument for this problem appears in the remark preceding Lemme~III \cite[p.~297, Remarque, $\ell.$14 onward]{Biernacki:1936}. However, we have not been able to fully justify the geometric step that identifies $Q$ as an angular set.
If Problem \ref{LA-problem} is valid, the remainder of the proof is straightforward.

We note that Lemma~\ref{lemma-closedness-LA} is relevant also to the classical equivalence between close-to-convexity and linear accessibility. The closedness of the class of linearly accessible functions plays an essential role in the proofs of the implication from close-to-convexity to linear accessibility given in \cite{Lewandowski:1960,BielLewa:1962}. A later proof in \cite{Koepf:1989} likewise makes use of the same closedness property.

\section*{Acknowledgements}
The authors would like to thank Professor Toshiyuki Sugawa and
Professor Hiroshi Yanagihara for their valuable comments and suggestions. 
They would also like to thank Professor Teodor Bulboac\u{a} for bringing the paper \cite{MR948445} to their attention.

	%	++++++++++++++++++++++++++++++++++++++++++++++++++++++
	%
	%		References
	%
	%	++++++++++++++++++++++++++++++++++++++++++++++++++++++

\bibliographystyle{amsalpha}
\def\cprime{$'$}
\providecommand{\bysame}{\leavevmode\hbox to3em{\hrulefill}\thinspace}
\providecommand{\MR}{\relax\ifhmode\unskip\space\fi MR }
\providecommand{\MRhref}[2]{%
  \href{http://www.ams.org/mathscinet-getitem?mr=#1}{#2}
}
\providecommand{\href}[2]{#2}

\end{document}